\documentclass[11pt,reqno]{amsart}
\usepackage{amsmath,amssymb,geometry,color,tikz-cd}
\usepackage[colorlinks=true, linkcolor=red!80!black, urlcolor=purple, citecolor=blue!70!black]{hyperref}
\usepackage{amssymb}

\usepackage{bbm}
\usepackage{enumitem}
\usepackage{upgreek}
\usepackage{bm}
\usepackage{stmaryrd}
\usepackage{mathrsfs}
\allowdisplaybreaks

\newcommand{\cC}{\mathcal{C}}
\newcommand{\cD}{\mathcal{D}}

\newcommand{\cF}{\mathcal{F}}
\newcommand{\cG}{\mathcal{G}}
\newcommand{\cH}{\mathcal{H}}

\newcommand{\cL}{\mathcal{L}}
\newcommand{\cM}{\mathcal{M}}
\newcommand{\cN}{\mathcal{N}}
\newcommand{\cO}{\mathcal{O}}

\newcommand{\cS}{\mathcal{S}}

\newcommand{\cX}{\mathcal{X}}
\newcommand{\cY}{\mathcal{Y}}

\newcommand{\fa}{\mathfrak{a}}
\newcommand{\fb}{\mathfrak{b}}

\newcommand{\fm}{\mathfrak{m}}

\newcommand{\bP}{\mathbb{P}}
\newcommand{\bC}{\mathbb{C}}
\newcommand{\bR}{\mathbb{R}}
\newcommand{\bA}{\mathbb{A}}
\newcommand{\bQ}{\mathbb{Q}}
\newcommand{\bZ}{\mathbb{Z}}
\newcommand{\bD}{\mathbb{D}}

\newcommand{\bG}{\mathbb{G}}
\newcommand{\bF}{\mathbb{F}}

\newcommand{\bN}{\mathbb{N}}

\newcommand{\hvol}{\widehat{\mathrm{vol}}}

\newcommand{\cI}{\mathcal{I}}

\newcommand{\Bs}{\mathrm{Bs}\,}

\DeclareMathOperator{\lct}{lct}

\DeclareMathOperator{\vol}{vol}

\DeclareMathOperator{\Pic}{Pic}

\DeclareMathOperator{\ord}{ord}

\DeclareMathOperator{\Proj}{Proj}

\DeclareMathOperator{\sing}{sing}

\DeclareMathOperator{\rk}{rk}

\newcommand{\ind}{\mathrm{ind}}
\newcommand{\tX}{\widetilde{X}}
\newcommand{\tC}{\widetilde{C}}
\newcommand{\tS}{\widetilde{S}}

\newcommand{\edim}{\textrm{edim}}
\newcommand{\tH}{\widetilde{H}}
\newcommand{\K}{\mathrm{K}}
\newcommand{\GIT}{\mathrm{GIT}}
\newcommand{\coeff}{\mathrm{coeff}}
\newcommand{\e}{\mathrm{e}}

\numberwithin{equation}{section}

\newtheorem{prop} {Proposition} [section]
\newtheorem{thm}[prop] {Theorem} 

\newtheorem{lem}[prop] {Lemma}
\newtheorem{cor}[prop]{Corollary}
\newtheorem{prop-def}[prop]{Proposition-Definition}

\newtheorem{conj}[prop]{Conjecture}

\newtheorem{thm-defn}[prop]{Theorem-Definition}

\theoremstyle{definition}
 
\newtheorem{rem}[prop] {Remark} 

\newtheorem{defn}[prop]{Definition}

\title{K-stability of cubic fourfolds}

\author{Yuchen Liu}
\address{Department of Mathematics, Northwestern University, Evanston, IL 60208, USA.}
\email{yuchenl@northwestern.edu}

\date{\today} 

\begin{document}
\begin{abstract}
We prove that the K-moduli space of cubic fourfolds is identical to their GIT moduli space. More precisely, the K-(semi/poly)stability of cubic fourfolds coincide to the
corresponding GIT stabilities, which was studied in detail by Laza. In particular, this
implies that all smooth cubic fourfolds admit K\"ahler-Einstein metrics. Key ingredients are local volume estimates in dimension three due to Liu-Xu, and Ambro-Kawamata's non-vanishing theorem for Fano fourfolds.
\end{abstract}

\maketitle


\section{Introduction}

K-stability is an algebro-geometric stability condition introduced by Tian \cite{Tia97} and later reformulated algebraically by Donaldson \cite{Don02} to detect the existence of K\"ahler-Einstein (KE) metrics on Fano varieties.  The
Yau-Tian-Donaldson (YTD) Conjecture predicts that the existence of a KE metric on a Fano variety $X$ is equivalent to the K-polystability of $X$. The relatively easier direction of the YTD Conjecture that KE metrics implies K-polystability was confirmed in \cite{Ber16}. When $X$ is smooth, the YTD Conjecture was proved in the celebrated works \cite{CDS15} and \cite{Tia15} using Cheeger-Colding-Tian theory. Later, a different approach to the YTD Conjecture, namely the variational approach, has been developed. Combining the analytic works \cite{BBJ18, LTW19, Li19} and the algebraic work \cite{LXZ21}, this approach gives a full proof of the YTD Conjecture for all (possibly singular) Fano varieties. However, it is often a challenging problem to check K-(semi/poly)stability of an explicit Fano variety. 

In recent years, the algebraic study of K-stability has successfully led to a new theory, known as the \emph{K-moduli theory}, that produces an algebraic construction of compact moduli spaces of Fano varieties. The Fano K-moduli theorem, proved in a combination of works \cite{Jia17, LWX18, CP18, BX18, ABHLX19, BLX19, Xu19, XZ19, XZ20, BHLLX20, LXZ21}, states that given dimension $n$ and volume $V$, there exists an Artin stack $\cM_{n,V}^{\rm Kss}$ of finite type parametrizing K-semistable Fano varieties, called the \emph{K-moduli stack}, and $\cM_{n,V}^{\rm Kss}$ admits a projective good moduli space $M_{n,V}^{\rm Kps}$ parametrizing K-polystable Fano varieties, called the \emph{K-moduli space}. In the $\bQ$-Gorenstein smoothable case, the Fano K-moduli theorem was proved earlier in \cite{LWX19, XZ19} (see also \cite{Oda15, SSY16}) based on analytic results from \cite{CDS15, Tia15}.

The K-moduli theory not only lays the foundation for these K-moduli spaces, but also provides strong tools to verify K-stability for explicit Fano varieties (e.g.\ Fano hypersurfaces or complete intersections). One notable strategy along this direction, namely the \emph{moduli continuity method}, goes as follows.
Using $\alpha$-invariants and group actions, it is usually easy to find a K-stable Fano manifold $X$ in the given family (e.g.\ Fermat hypersurfaces \cite{Tia00, AGP06, Zhu20}).  Then by openness of K-(semi)stability \cite{BL18b, BLX19, Xu19, LXZ21}, there exists an open neighborhood of $[X]$ in the parameter space which parametrizes K-stable Fano varieties $\cX_t$. By the Fano K-moduli theorem, there exists a component  $M^{\K}$ of the K-moduli space $M_{n,V}^{\rm Kps}$ that parametrizes those $\cX_t$ and their K-polystable $\bQ$-Gorenstein limits, where $n=\dim(X)$ and $V=(-K_X)^n$. Hence the K-moduli space $M^{\K}$ is birational to the GIT moduli space $M^{\GIT}$. Using the local-to-global volume comparison in \cite{Liu18}, if the global volume of $X$ is relatively large, then we can often get a good control of the singularities appearing in the boundary of $M^{\K}$. This enables us to give an explicit description of the birational map $M^{\K}\dashrightarrow M^{\GIT}$, and in some cases to even show that it is an isomorphism. This strategy first appeared implicitly in \cite{Tia90} where Tian showed that all smooth del Pezzo surfaces with reductive automorphism groups are K-polystable. Later, it was used to construct explicit K-moduli compactifications of del Pezzo surfaces of degree 4 in  \cite{MM93}  and of degree $\leq 3$ in \cite{OSS16} where the latter work was more focused on the stability study. In higher dimensions, it is shown that $M^{\K}$ is isomorphic to $M^{\GIT}$ for complete intersections of two quadric hypersurfaces in \cite{SS17} (where smooth ones were shown to be K-stable earlier in \cite{AGP06}) and cubic threefolds in \cite{LX19}. Some cases for log Fano pairs have been worked out as well, see e.g.\ \cite{Fuj17, GMGS, ADL19, ADL20, ADL21}. Despite these results, much less is known in higher dimensions.

In this paper, we carry out this strategy for cubic fourfolds by showing that their K-moduli space is isomorphic to their GIT moduli space. In other words, the K-(semi/poly)stability of cubic fourfolds are the same as their GIT (semi/poly)stability. 

\begin{thm}\label{thm:K=GIT}
Let $X\subset\bP^5$ be a (possibly singular) cubic hypersurface. Then $X$ is K-(semi/poly)stable if and only if $X\subset\bP^5$ is GIT (semi/poly)stable. In particular, the K-moduli space $M^{\K}$ parametrizing K-polystable $\bQ$-Fano varieties admitting $\bQ$-Gorenstein smoothings to smooth cubic fourfolds is isomorphic to the GIT moduli space $M^{\GIT}$ of cubic fourfolds.
\end{thm}

Note that the ``only if'' direction in Theorem \ref{thm:K=GIT} follows from the general fact in \cite{Tia94, PT06} that a K-(semi/poly)stable Fano hypersurface is always GIT (semi/poly)stable. On the other hand, the ``if'' direction for general Fano hypersurfaces $X\subset \bP^{n+1}$ is expected to be true only when $\deg X = 3$. In fact, if $\deg X\geq 4$ then there are non-reduced, hence K-unstable, GIT polystable hypersurfaces $X$, e.g.\ a multiple of a smooth hyperquadric. In addition, the ``if'' direction can also fail for non-hypersurface Fano varieties, e.g.\ quartic double solids \cite[Theorem 1.4]{ADL21}.

The GIT of cubic fourfolds was studied in detail by Laza \cite{Laz09} (see also \cite{Yok08}). As a consequence, we have the following result which, together with \cite{CDS15, Tia15}, implies that any smooth cubic fourfold admits a KE metric. We also obtain a result on  singularities of GIT semistable cubic fourfolds without involving direct GIT calculation, which answers affirmatively a question of Spotti and Sun  \cite[Question 5.8]{SS17} in dimension $4$.

\begin{cor}\label{cor:K-cubic}\leavevmode
\begin{enumerate}
    \item All smooth cubic fourfolds are K-stable.
    \item All cubic fourfolds with simple singularities are K-stable.
    \item All GIT polystable cubic fourfolds are K-polystable. For a list of generic singularities of GIT polystable cubic fourfolds with non-simple singularities, see \cite[Theorem 1.2 and Table 3]{Laz09}.
    \item Any GIT semistable cubic fourfold has Gorenstein canonical singularities.
\end{enumerate}
In particular, each cubic fourfold in (1), (2) or (3) admits a (weak) KE metric.
\end{cor}

We note that combined Corollary \ref{cor:K-cubic} with \cite{Che01, CP02} and \cite[Corollary 1.4]{Fuj19} on K-stability of smooth quintic fourfolds and the very recent result \cite[Theorem 1.1]{AZ20} on K-stability of smooth quartic fourfolds, we answer affirmatively the
folklore conjecture that all smooth Fano hypersurfaces have KE metrics in dimension $4$.

In the process of proving Theorem \ref{thm:K=GIT}, we confirm the ODP Gap conjecture of local volumes \cite[Conjecture 5.5]{SS17} (see also Conjecture \ref{conj:ODP}) for all local complete intersection singularities. The proof uses the lower semicontinuity of local volumes \cite{BL18a}. We note that this conjecture was confirmed in dimension at most $3$ \cite{LL16, LX19}.

\begin{thm}[=Theorem \ref{thm:vol-lci}]\label{thm:lci}
 Let $(x\in X)$ be an $n$-dimensional non-smooth local complete intersection klt singularity. Then
 \[
  \hvol(x,X)\leq 2(n-1)^n,
 \]
and equality holds if and only if it is an ordinary double point.
\end{thm}

Our proof of Theorem \ref{thm:K=GIT} starts from parallel arguments in \cite{LX19} where the author and Xu showed the similar result to Theorem \ref{thm:K=GIT} for cubic threefolds. Suppose $X$ is  an $n$-dimensional K-semistable $\bQ$-Fano variety admitting a $\bQ$-Gorenstein smoothing to cubic hypersurfaces $\cX_t$. Let $L$ be the $\bQ$-Cartier Weil divisor on $X$ as the limit of hyperplane sections of $\cX_t$. By the moduli continuity method in \cite{SS17, LX19}, the analogous result to Theorem \ref{thm:K=GIT} for cubic $n$-folds would follow from showing $X$ is a (possibly singular) cubic hypersurface. By \cite{Fuj90} this reduces to showing $L$ is Cartier. Using the local-to-global volume comparison from \cite{Liu18}, the divisor $L$ being Cartier would follow from the ODP Gap Conjecture of  local volumes in dimension $n$ (see Conjecture \ref{conj:ODP}), especially applied to the index $1$ cover of a singular point $x\in X$ where $L$ is not Cartier. In \cite{LX19}, the author and Xu verified the ODP Gap conjecture in dimension $3$, thus proving that K-(semi/poly)stability of cubic threefolds coincide with corresponding GIT stabilities. Such an approach depends heavily on the classification of terminal and canonical singularities in dimension $3$, and is currently out of reach in dimension $4$ or higher.

To overcome this difficulty, we study the explicit geometry of the linear systems $|L|$ and $|2L|$ on  $X$ under the assumption that $L$ is not Cartier. Using the local-to-global volume estimates from \cite{Liu18, LX19} and a Bertini type result for local volumes (see Theorem \ref{thm:bertini}), we show that $2L$ is Cartier, and $L$ is Cartier away from finitely many points. Then using Ambro-Kawamata's non-vanishing theorem for fundamental divisors on Fano varieties with large Fano index \cite{Amb99, Kaw00}, we show that for general elements $D\in |2L|$ and $H\in |L|$, the complete intersection $(G=D\cap H, L|_G)$ is a polarized K3 surface with Du Val singularities of degree $6$. Then classical results on linear systems of K3 surfaces \cite{May72} implies that $|2L|$ is base point free, and $(G,L|_G)$ is either a complete intersection, hyperelliptic, or 
unigonal. In the first case, we show that the index $1$ cover  of a non-Cartier point $x\in X$ of $L$ must be a local complete intersection which satisfies the ODP Gap Conjecture (see Theorem \ref{thm:lci}). Thus similar arguments to \cite{LX19}  show that $X$ is K-unstable. In the last two cases, by analyzing the rational map $\phi_{|L|}:X\dashrightarrow\bP^5$ we show that $\alpha(X)<\frac{1}{5}$ which implies that $X$ is K-unstable by \cite[Theorem 3.5]{FO16}.  Given these contradictions, we conclude that $L$ is Cartier on $X$, hence proving Theorem \ref{thm:K=GIT}.

\subsection*{Acknowledgements}

I would like to thank Chenyang Xu and Ziquan Zhuang for fruitful discussions and helpful comments on a draft, including Remark 
\ref{rem:zhuang} and a simplification of the proof of Proposition \ref{prop:K3-hyp}. I would like to thank Kento Fujita, Chen Jiang, Zhiyuan Li, Linquan Ma, Yuji Odaka, Giulia Sacc\`a, Cristiano Spotti, and Gang Tian for helpful discussions and comments. The author was partially supported by  the NSF Grant No. 2148266 (transferred from NSF Grant No. DMS-2001317).

\section{K-stability and local volumes}\label{sec:k-vol}

Throughout this paper, we work over the field $\bC$. We follow the standard convention from \cite{KM98, Kol13}. A pair $(X,\Delta)$ is a normal variety $X$ together with an effective $\bQ$-divisor $\Delta$ such that $K_X+\Delta$ is $\bQ$-Cartier. A pair $(X,\Delta)$ is called a log Fano pair if $X$ is projective, $(X,\Delta)$ is klt, and $-K_X-\Delta$ is ample. We call $X$ a $\bQ$-Fano variety if $(X,0)$ is a log Fano pair. We call $X$ a lc Fano variety if $X$ is normal projective with log canonical singularities, and $-K_X$ is $\bQ$-Cartier ample. A klt singularity $x\in X$ is a closed point $x$ on a normal variety $X$ with klt singularities. 

\subsection{Valuative criteria for K-stability}

\begin{defn}
Let $X$ be a normal variety. A prime divisor $E$ over $X$ is a prime divisor $E$ on a normal variety $Y$ together with a proper birational morphism $\mu:Y\to X$. The center of $E$ on $X$ is $\mu(E)$. Moreover, if $K_X$ is $\bQ$-Cartier, then we define the log discrepancy of $E$ as \[
A_X(E):= 1 + \coeff_E(K_Y-\mu^* K_X). 
\]
From the definition we know that $A_X(E)>0$ (resp. $\geq 0$) if $X$ has klt (resp. log canonical) singularities.
\end{defn}

\begin{defn}
Let $X$ be an $n$-dimensional $\bQ$-Fano variety. Let $E$ be a prime divisor over $E$. The pseudo-effective threshold of $E$ is defined as 
\[
T_X(E):= \sup \{t\in \bR\mid \mu^*(-K_X)-t E \textrm{ is big}\}.
\]
The $S$-invariant of $E$, first introduced in \cite{BJ17}, is defined as 
\[
S_X(E):= \frac{1}{(-K_X)^n}\int_0^{T_X(E)}\vol_Y(\mu^*(-K_X)-tE) dt.
\]
The $\beta$-invariant of $E$, first introduced in \cite{Fuj16}, is defined as 
\[
\beta_X(E):=A_X(E)-S_X(E).
\]
\end{defn}

The original definition of K-(poly/semi)stability introduced by \cite{Tia97, Don02} is by checking the sign of generalized Futaki invariants of test configurations. In this paper, we will use the valuative criterion for K-(semi)stability invented by Fujita \cite{Fuj16} and Li \cite{Li17} with complementary result by Blum and Xu \cite{BX18}. 

Let $X$ be a $\bQ$-Fano variety (resp. a lc Fano variety) with $-rK_X$ Cartier for $r\in \bZ_{>0}$.  Recall from \cite{LX14, LWX18} that a test configuration $(\cX,\cL)/\bA^1$ of $(X,-rK_X)$ is special (resp. weakly special) if $(\cX,\cX_0)$ is plt (resp. log canonical) and $\cL\sim_{\bQ} -rK_{\cX/\bA^1}$, and we say that $\cX_0$ is a special degeneration (resp. a weakly special degeneration) of $X$.

\begin{thm-defn}\label{thm:valuative}
Let $X$ be a $\bQ$-Fano variety. Then 
\begin{enumerate}
    \item \cite{Fuj16, Li17} $X$ is K-semistable if and only if $\beta_X(E)\geq 0$ for any prime divisor $E$ over $X$;
    \item \cite{BX18} $X$ is K-stable if and only if $\beta_X(E)> 0$ for any prime divisor $E$ over $X$;
    \item \cite{LWX18} $X$ is K-polystable if and only if any K-semistable special degeneration of $X$ is isomorphic to itself.
    \item $X$ is K-unstable if it is not K-semistable.
\end{enumerate}
\end{thm-defn}

\begin{defn}
Let $X$ be a $\bQ$-Fano variety. The $\alpha$-invariant of $X$ is defined as 
\[
\alpha(X)=\inf \{\lct(X;D)\mid 0\leq D\sim_{\bQ}-K_X\}.
\]
By \cite{BJ17}, we know that $\alpha(X)=\inf_E \frac{A_X(E)}{T_X(E)}$ where the infimum is taken over all prime divisors $E$ over $X$. 
\end{defn}

\begin{thm}[{\cite[Theorem 3.5]{FO16}}]\label{thm:alpha-lowerbound}
Let $X$ be an $n$-dimensional K-semistable $\bQ$-Fano variety. Then we have $\alpha(X)\geq \frac{1}{n+1}$.
\end{thm}

\subsection{Local volumes}

In this section, we recall the concept of normalized volume of valuations over a klt singularity first introduced by C. Li \cite{Li18}. For simplicity, we restrict ourselves to divisorial valuations.

\begin{defn}[\cite{Li18}]
Let $x\in X$ be an $n$-dimensional klt singularity. For a prime divisor $E$ over $X$ centered at $x$, we define (following \cite{ELS03}) the volume of $E$ over $(x\in X)$ to be 
\[
\vol_{X,x}(E):= \lim_{m\to \infty}\frac{\ell(\cO_{X,x}/\fa_m(E))}{m^n/n!}.
\]
Here $\fa_m(E):=\{f\in \cO_{X,x}\mid \ord_E(f)\geq m\}$ and $\ell$ denotes the length of an Artinian module. The normalized volume of $E$ over $(x\in X)$ is defined as
\[
\hvol_{X,x}(E):= A_X(E)^n\cdot \vol_{X,x}(E). 
\]
The local volume of $x\in X$ is defined as 
\[
\hvol(x,X):= \inf_{E}\hvol_{X,x}(E),
\]
where the infimum runs over all prime divisors $E$ over $X$ centered at $x$.
\end{defn}

There are alternative characterizations of local volumes. We provide two of them using ideals and Koll\'ar components which are useful.

\begin{thm}[\cite{Liu18, Blu18}]\label{thm:vol=lcte}
Let $x\in X$ be an $n$-dimensional klt singularity. Denote by $(R,\fm):=(\cO_{X,x}, \fm_{X, x})$. Then we have 
\[
\hvol(x,X)=\inf_{\fa\colon \textrm{$\fm$-primary}} \lct(X;\fa)^n\cdot \e(\fa)=\min_{\fa_{\bullet}\colon \textrm{$\fm$-primary}} \lct(X;\fa_{\bullet})^n\cdot \e(\fa_{\bullet}).
\]
Here $\fa$ (resp. $\fa_{\bullet}$) represents an ideal (resp. a multiplicative graded sequence of ideals) of $R$, and $\e$ denotes the Hilbert-Samuel multiplicity.
\end{thm}

\begin{defn}
Let $x\in X$ be a klt singularity. We say that a proper birational morphism $\mu: Y\to X$ from a normal variety $X$ provides a Koll\'ar component $S$ over $(x\in X)$ if $\mu$ is an isomorphism over $X\setminus\{x\}$, the preimage $\mu^{-1}(x)=S$ is a $\bQ$-Cartier prime divisor on $Y$,  the pair $(Y, S)$ is plt, and $-S$ is $\mu$-ample.
\end{defn}

\begin{thm}[\cite{LX20}]\label{thm:lx-Kollarcomp}
For any klt singularity $x\in X$, we have $\hvol(x,X)=\inf_{S} \hvol_{X,x}(S)$ where the infimum runs over all Koll\'ar components $S$ over $X$ centered at $x$.
\end{thm}

The following conjecture was asked in \cite{SS17}. It was confirmed in dimension $2$ and $3$ by \cite[Proposition 4.10]{LL16} and \cite[Theorem 1.3]{LX19} respectively. 

\begin{conj}[ODP Gap Conjecture]\label{conj:ODP}
Let $(x\in X)$ be an $n$-dimensional non-smooth klt singularity. Then 
\[
\hvol(x,X)\leq 2(n-1)^n,
\]
and equality holds if and only if it is an ordinary double point.
\end{conj}

\begin{thm}\label{thm:odp-lowdim}
\begin{enumerate}
    \item \cite{LL16, LX19} Conjecture \ref{conj:ODP} holds when $n\leq 3$. 
    \item \cite[Theorem 1.6]{LX19} Let $x\in X$ be an $n$-dimensional klt singularity. Then we have $\hvol(x,X)\leq n^n$, and equality holds if and only if it is smooth.
\end{enumerate}

\end{thm}

The following result verifies Conjecture \ref{conj:ODP} for local complete intersection singularities.

\begin{thm}\label{thm:vol-lci}
Conjecture \ref{conj:ODP} holds for all local complete intersection singularities.
\end{thm}

\begin{proof}
Let $x\in X$ be an $n$-dimensional local complete intersection singularity. Since Conjecture \ref{conj:ODP} holds in dimension $\leq 3$ by Theorem \ref{thm:odp-lowdim}(1), we may assume that $n\geq 4$. Then there exists a locally closed immersion $X\hookrightarrow Z$ into a smooth variety $Z$ of dimension $n+r$ such that $X=V(f_1, f_2,\cdots, f_r)$ for some $f_i\in \cO_{Z, x}$. We may assume that $r$ achieves its minimum, i.e.\ $r:=\edim(x,X)-n$ where $\edim(x,X):=\ell(\fm_{X,x}/\fm_{X,x}^2)$ denotes the embedding dimension. In particular, $\ord_x(f_i)\geq 2$ for any $1\leq i\leq r$. Since $x\in X$ is non-smooth, we have $r\geq 1$. 

If $r\geq 2$, then we choose $g_1,g_2\in \fm_{Z,x}^2$ and $g_3, \cdots, g_n\in \fm_{Z,x}$, so that $\bar{g}_1,\bar{g}_2\in  \fm_{Z,x}^2/\fm_{Z,x}^3$ and $\bar{g}_i\in \fm_{Z,x}/\fm_{Z,x}^2$ are general for $i\geq 3$. Let $X_t: = V(f_1+tg_1, f_2+tg_2, \cdots, f_r+tg_r)$. Then it is clear that $(x\in X_t)_{t\in \bA^1}$ is a $\bQ$-Gorenstein flat family of klt singularities. Moreover, for a general $t\in \bA^1$ we know that $x\in X_t$ is a local complete intersection of two quadrics with smooth projective tangent cone. Let $S_t$ be the exceptional divisor of the ordinary blow up of $x\in X_t$ for general $t$. Then simple calculation shows that $A_{X_t}(S_t)=  n-2$ and $\vol_{X_t, x}(S_t)=\e(x,X_t)=4$. Thus for general $t$ we have
\[
 \hvol(x,X_t)\leq A_{X_t}(S_t)^n\cdot \vol_{X_t, x}(S_t)= 4(n-2)^n <2(n-1)^n. 
\]
Thus the lower semicontinuity of local volumes \cite{BL18a} implies that $\hvol(x,X)\leq \hvol(x,X_t)< 2(n-1)^n$. 

If $r=1$, then $x\in X= V(f)$ is a hypersurface singularity for some $f\in \cO_{Z,x}$. By \cite[Lemma 3.1]{LX19}, we know that $\hvol(x,X)\leq (n+1-\ord_x f)^n\cdot \ord_x f\leq  2(n-1)^n$. If the equality holds, then we must have $\ord_x f=2$. Assume to the contrary that $x\in X$ is not an ordinary double point. By choosing a suitable algebraic coordinates $(z_0,\cdots, z_n)$ at $x\in Z$,  we may assume that $f-(z_0^2+\cdots+z_m^2)\in \fm_{Z,x}^3$ where $m<n$. Let $g=z_0^2+\cdots+z_{n-1}^2+z_n^3\in \fm_{Z,x}^2$, and let $X_t:=V(f+tg)$. Then by \cite{BL18a} we know that $\hvol(x,X)\leq \hvol(x,X_t)$ for general $t$. Since the degree $2$ part of $f+tg$ has rank $n$ for general $t$, we know that $x\in X_t$ is an $n$-dimensional $A_{2}$-singularity. Hence \cite[Example 5.3]{Li18} for $n\geq 4$ implies that 
\[
 \hvol(x,X)\leq \hvol(x,X_t)\leq \frac{2n^n(n-2)^{n-1}}{(n-1)^{n-1}}< 2(n-1)^n. 
\]
Hence we get a contradiction. This finishes the proof.
\end{proof}

The following result from \cite[Theorem 1.3]{XZ20} on finite degree formula of local volumes is very useful. Note that when $X$ is a Gromov-Hausdorff limit of K\"ahler-Einstein Fano manifolds, such a result was proven earlier in \cite[Theorem 1.7]{LX18}.

\begin{thm}[\cite{XZ20}]\label{thm:vol-finitedeg}
Let $\tau: (\tilde{x}\in \tX)\to (x\in X)$ be a finite quasi-\'etale Galois morphism between klt singularities. Then 
\[
\hvol(\tilde{x},\tX)=\deg(\tau)\cdot \hvol(x,X).
\]
\end{thm}

The following theorem is one of the key ingredients in the moduli continuity method. It is a generalization of \cite{Fuj18}.

\begin{thm}[\cite{Liu18}]\label{thm:liu18}
Let $X$ be an $n$-dimensional K-semistable $\bQ$-Fano variety. Then for any closed point $x\in X$ we have 
\[
(-K_X)^n\leq \left(1+\frac{1}{n}\right)^n \hvol(x,X).
\]
\end{thm}

The following lemma is well-known to experts.

\begin{lem}\label{lem:volsame}
Let $(x\in X)$ and $(x'\in X')$ be two klt singularities that are analytically isomorphic, that is, $\widehat{\cO_{X,x}}\cong \widehat{\cO_{X',x'}}$. Then $\hvol(x,X)=\hvol(x',X')$. 
\end{lem}

\begin{proof}
Denote by $(R,\fm):=(\cO_{X,x},\fm_{X,x})$ and $(R',\fm'):=(\cO_{X',x'},\fm_{X',x'})$. Denote by $\iota: \widehat{R}\to \widehat{R'}$ the ring isomorphism. Then any $\fm$-primary ideal $\fa$ of $R$ corresponds to a unique $\fm'$-primary ideal $\fa'$ of $R'$ via $\fa'=\iota(\widehat{\fa})\cap R'$. Under this correspondence, it is easy to see $\e(\fa)=\e(\fa')$. By \cite[Proposition 2.11]{dFEM11} we know that $\lct(X;\fa)=\lct(X';\fa')$. Hence the statement follows from Theorem \ref{thm:vol=lcte}.
\end{proof}

\subsection{Bertini type result for local volumes}
In this section, we prove the following Bertini type result for local volumes.

\begin{thm} \label{thm:bertini}
Let $x\in X$ be a non-isolated $n$-dimensional klt singularity. Then there exists a non-smooth $(n-1)$-dimensional klt singularity $y\in Y$ such that 
\[
\frac{\hvol(x,X)}{n^n}\leq \frac{\hvol(y,Y)}{(n-1)^{n-1}}.
\]
Moreover, $Y$ can be chosen as a general hyperplane section of $X$. 
\end{thm}

Before presenting the proof of Theorem \ref{thm:bertini}, we need the following result on the local volume of fibrations.

\begin{prop}\label{prop:vol-fibration}
Let $\pi:\cX\to B$ together with a section $\sigma:B\to \cX$ be a $\bQ$-Gorenstein flat family of klt singularities over a smooth curve $B$. Then for a general point $b\in B$ we have 
\[
\frac{\hvol(\sigma(b), \cX)}{n^n}=\frac{\hvol(\sigma(b), \cX_b)}{(n-1)^{n-1}}
\]
\end{prop}

\begin{proof}
By the adjunction of local volumes \cite[Theorem 1.7]{LZ18}, we always have $\frac{\hvol(\sigma(b), \cX)}{n^n}\geq \frac{\hvol(\sigma(b), \cX_b)}{(n-1)^{n-1}}$ for any $b\in B$. Thus it suffices to show the reverse inequality for general $b\in B$.

Denote by $\bar{\eta}$ the geometric generic point of $B$. Then by \cite{Xu19} we know that there for a general point $b\in B$ we have $\hvol(\sigma(b),\cX_b)=\hvol(\sigma(\bar{\eta}), \cX_{\bar{\eta}})$. Let us fix an arbitrary $\epsilon>0$. By Theorem \ref{thm:lx-Kollarcomp}, there exists a Koll\'ar component $\cS_{\bar{\eta}}$ over $\sigma(\bar{\eta})\in \cX_{\bar{\eta}}$ such that 
\begin{equation}\label{eq:geom.gen.fiber}
\hvol_{\cX_{\bar{\eta}},\sigma(\bar{\eta})}(\cS_{\bar{\eta}})\leq \hvol(\sigma(\bar{\eta}), \cX_{\bar{\eta}})+\epsilon.
\end{equation}
Denote by $\mu_{\bar{\eta}}: \cY_{\bar{\eta}}\to \cX_{\bar{\eta}}$ the plt blow up extracting $\cS_{\bar{\eta}}$. Hence there exists an \'etale morphism $\widetilde{B}\to B$ such that $\mu_{\bar{\eta}}$ extends over $\cX\times_B \widetilde{B}$ which provides a flat family of Koll\'ar components (see \cite[Definition A.2]{LX19}). Since the local volume is preserved under \'etale morphism by Lemma \ref{lem:volsame}, we may replace $B$ by $\widetilde{B}$. Thus there is a birational morphism $\mu: \cY\to \cX$ which provides a flat family $\cS$ of Koll\'ar components over $\cX$ centered at $\sigma(B)$. Denote by $\fb_m:= \mu_*\cO_{\cY}(-m\cS)$. Since $-\cS$ is a $\mu$-ample $\bQ$-Cartier divisor, we know that $\fb_m$ is a flat family of ideals over $B$ for $m\gg 1$ (i.e.\ $\cO_{\cX}/\fb_m$ is flat over $B$). After replacing $B$ with a dense open set, we may assume that $\fb_m$ is a flat family of ideals over $B$ for any $m\in \bZ_{\geq 0}$. For $p\in \bR_{\geq 0}$, we define $\fb_p:=\fb_{\lceil p\rceil}=\mu_*\cO_{\cY}(-\lceil p\cS\rceil)$. Let $\cS_b$ and $\fb_{p,b}$ be the restriction of $\cS$ and $\fb_p$ on $\cX_b$. By flatness of $\cS$ and $\fb_{\bullet}$ over $B$, we have $\fb_{p, b}=\fa_p(\cS_b)$ for any $b\in B$. 

Next, we estimate the local volume $\hvol(\sigma(b),\cX)$. Fix a general $b\in B$. Let $t\in \cO_{b,B}$ be its uniformizer. For any $s\in \bR_{>0}$, consider the ideal sequence $\cI_{\bullet, s}$ as
\[
\cI_{m,s}:=\fb_{ms}+\fb_{(m-1)s} t+ \fb_{(m-2)s} t^2 +\cdots + \fb_{s} t^{m-1}+ (t^m).
\]
By the definition of $\fb_p$ we know that $\cI_{\bullet,s}$ is a multiplicative ideal sequence of $\cO_{\cX,\sigma(b)}$ cosupported at $\sigma(b)$.  Then we know that $
\ell(\cO_{\cX,\sigma(b)}/\cI_{m,s})=\sum_{i=1}^m \ell(\cO_{\cX_b,\sigma(b)}/\fb_{is,b})$.
Since $\ell(\cO_{\cX_b,\sigma(b)}/\fb_{p,b})=\frac{1}{(n-1)!}\vol_{\cX_b,\sigma(b)}(\cS_b) p^{n-1}+ O(p^{n-2})$, we know that 
\[
\ell(\cO_{\cX,\sigma(b)}/\cI_{m,s})=\frac{1}{n!}\vol_{\cX_b,\sigma(b)}(\cS_b) s^{n-1} m^n +O(m^{n-1}).
\]
This implies that 
\begin{equation}\label{eq:mult-est}
\e(\cI_{\bullet, s})=\vol_{\cX_b,\sigma(b)}(\cS_b) s^{n-1}.
\end{equation} 
Let $v_{s}$ be the valuation of $\bC(\cX)$ as the quasi-monomial combination of $\cY_b$ and $\cS$ of weight $s$ and $1$, respectively. Then it is clear that $A_{\cX}(v_s) = s+  A_{\cX}(\cS)= s+ A_{\cX_b}(\cS_b)$ and $v_s(\cI_{m,s})\geq ms$. Hence we have
\begin{equation}\label{eq:lct-est}
\lct(\cX;\cI_{\bullet,s})\leq \frac{A_{\cX}(v_s)}{v_s(\cI_{\bullet,s})} \leq 1+ s^{-1} A_{\cX_b}(\cS_b).
\end{equation}
Combining Theorem \ref{thm:vol=lcte}, \eqref{eq:mult-est} and \eqref{eq:lct-est}, we obtain
\[
\hvol(\sigma(b), \cX)\leq \lct(\cX;\cI_{\bullet,s})^n\cdot \e(\cI_{\bullet, s})\leq (1+ s^{-1} A_{\cX_b}(\cS_b))^n\cdot \vol_{\cX_b,\sigma(b)}(\cS_b) s^{n-1}.
\]
Since $s$ is arbitrary, we may choose $s=\frac{A_{\cX_b}(\cS_b)}{n-1}$ which minimizes the right-hand-side of the above inequality. Hence we have 
\[
\hvol(\sigma(b),\cX)\leq \frac{n^n}{(n-1)^{n-1}} \hvol_{\cX_b,\sigma(b)}(\cS_b)\leq \frac{n^n}{(n-1)^{n-1}}(\hvol(\sigma(\bar{\eta}), \cX_{\bar{\eta}})+\epsilon).
\]
Here the second inequality follows from \eqref{eq:geom.gen.fiber} and $\hvol_{\cX_b,\sigma(b)}(\cS_b)=\hvol_{\cX_{\bar{\eta}},\sigma(\bar{\eta})}(\cS_{\bar{\eta}})$ by flatness of $\cS$ and $\fb_\bullet$ over $B$. By \cite{Xu19} we know that $\hvol(\sigma(\bar{\eta}),\cX_{\bar{\eta}})=\hvol(\sigma(b),\cX_b)$ for general $b$. Hence by letting $\epsilon\to 0$, we prove the reverse inequality that $\frac{\hvol(\sigma(b),\cX)}{n^n}\leq \frac{\hvol(\sigma(b),\cX_b)}{(n-1)^{n-1}}$ for general $b\in B$. The proof is finished.
\end{proof}

\begin{proof}[Proof of Theorem \ref{thm:bertini}] For simplicity, we assume that $X$ is affine. Let $C\subset X_{\sing}$ be an integral curve through $x$. Let $\pi: X\to \bA^1$ be a general linear projection. Then $\pi|_C: C\to \bA^1$ is quasi-finite. Let $y\in C$ be a general point. Let $Y$ be the fiber of $\pi$ containing $y$. Then after taking base change of $\pi$ to the normalization of $C$, Lemma \ref{lem:volsame} and Proposition \ref{prop:vol-fibration} implies that $\frac{\hvol(y,X)}{n^n}=\frac{\hvol(y,Y)}{(n-1)^{n-1}}$.
Since $y\in C$ is general, we have $\hvol(x,X)\leq \hvol(y,X)$ by \cite{BL18a}. This finishes the proof.
\end{proof}

\begin{cor}\label{cor:largesing}
Let $x\in X$ be a non-smooth klt singularity of dimension $n\geq 3$. Assume that $\dim_x X_{\sing}\geq n-3$. Then $\hvol(x,X)/n^n\leq 16/27$. In particular, Conjecture \ref{conj:ODP} holds for $x\in X$. 
\end{cor}

\begin{proof}
 We focus on the first inequality $\hvol(x,X)/n^n\leq 16/27$, as the  statement on Conjecture \ref{conj:ODP} is a consequence of this inequality by a simple computation
 $ \frac{16}{27}n^n < 2(n-1)^n$ whenever $n\geq 4$, and the conjecture holds in dimension $3$ by Theorem \ref{thm:odp-lowdim}(1).

 We do induction on $n\geq 3$. When $n=3$, the first inequality is precisely Theorem \ref{thm:odp-lowdim}(1). Assume that the first inequality is true in dimension $n-1$ with $n\geq 4$. Let $x\in X$ be a non-smooth klt singularity of dimension $n$ with $\dim_x X_{\sing}\geq n-3$. Let $V$ be an irreducible component of $X_{\sing}$ such that $x\in V$ and $\dim V\geq n-3$. Let $C\subset V$ be an integral curve through $x$. Then the proof of Theorem \ref{thm:bertini} implies that there exists  a hyperplane section $Y\subset X$ and a closed point $y\in Y\cap C$ such that $\frac{\hvol(x,X)}{n^n}\leq \frac{\hvol(y,Y)}{(n-1)^{n-1}}$. Furthermore, since $Y_{\sing}\supseteq X_{\sing}\cap Y\supseteq V\cap Y$, we know that 
 \[
 \dim_y Y_{\sing}\geq \dim_y (V\cap Y)\geq \dim_y V -1 \geq n-4.
 \]
 By induction hypothesis, we have $\frac{\hvol(y,Y)}{(n-1)^{n-1}}\leq \frac{16}{27}$. Hence we have
 \[
 \frac{\hvol(x,X)}{n^n}\leq \frac{\hvol(y,Y)}{(n-1)^{n-1}}\leq \frac{16}{27}.
 \]
 Thus the proof is finished.
\end{proof}

\section{Local-to-global volume estimates}

In order to prove our main result Theorem \ref{thm:K=GIT}, we follow the strategy from \cite{SS17, LX19}, that is, to show that any K-semistable $\bQ$-Gorenstein limits of cubic hypersurfaces is again a cubic hypersurface. The following result on K-semistable degeneration of higher dimensional cubic hypersurfaces is an easy consequence of arguments therein.

\begin{thm}\label{thm:LX-constrain}
Let $n\geq 4$ be an integer.
Let $\cX\to B$ be a K-semistable $\bQ$-Fano family over a smooth pointed curve $0\in B$ such that over $B^\circ:=B\setminus \{0\}$ it is a smooth family of cubic $n$-folds. Then there exists a $\bQ$-Cartier integral Weil divisor class $L$ on $X:=\cX_0$ such that the following properties hold.
\begin{enumerate}
    \item $\cO_X(mL)$ is Cohen-Macaulay for any $m\in\bZ$.
    \item $-K_X\sim (n-1) L$ and $(L^n)=3$.
    \item $h^i(X,\cO_X(mL))=h^i(\cX_b,\cO_{\cX_b}(m))$ for any $m\in\bZ$, $i\geq 0$, and $b\in B^\circ$. Moreover, $h^0(X,\cO_X(L))=n+2$, and $h^j(X,\cO_X(mL))=0$ for any $m\in\bZ$ and $1\leq j\leq n-1$.
    \item Any $\bQ$-Cartier Weil divisor $D$ on $X$ satisfies that $2D$ is Cartier. In particular, $2L$ is Cartier.
\end{enumerate}
\end{thm}

\begin{proof}
Denote by $\cX^\circ:=\cX\setminus\cX_0$. By base change to a finite cover of $B$, we can find a hyperplane section $\cL^{\circ}\sim_B \cO_{\cX^{\circ}}(1)$ and taking Zariski closure yields a Weil divisor $\cL$ on $\cX$. It is clear that $-K_{\cX^\circ/B^\circ}\sim_{B^\circ} (n-1)\cL^\circ$. Since $\cX_0\sim_B 0$ is integral, we know that $-K_{\cX/B}\sim_{B} (n-1)\cL$ which implies that $\cL$ is $\bQ$-Cartier. By assumption $\cX$ has klt singularities, so \cite[Corollary 5.25]{KM98} implies that $\cO_{\cX}(m\cL)$ is Cohen-Macaulay for any $m\in\bZ$. Thus $L:=\cL|_{\cX_0}$ is a $\bQ$-Cartier Weil divisor satisfying that $\cO_{\cX_0}(mL)\cong \cO_{\cX}(mL)\otimes\cO_{\cX_0}$ is Cohen-Macaulay for any $m\in\bZ$, and $-K_{X}\sim (n-1)L$. The fact that $(L^n)=3$ comes from $(\cO_{\cX_b}(1)^n)=3$ and $\cL$ is $\bQ$-Cartier. Hence we have shown (1) and (2). 

For part (3), notice that $\cO_{\cX}(m\cL)$ is flat over $B$ for any $m\in\bZ$ whose fiber over $b$ and $0$ are $\cO_{\cX_b}(m)$ and  $\cO_{X}(mL)$ respectively. 
If $m\geq 2-n$, then $mL-K_X\sim (m+n-1)L$ and $m\cL_b - K_{\cX_b}\sim \cO_{\cX_b}(m+n-1)$ are both ample. Hence 
Kawamata-Viehweg vanishing implies that  $H^i(X,\cO_X(mL))=H^i(\cX_b, \cO_{\cX_b}(m))=0$ for any $i\geq 1$ and $m\geq 2-n$. By flatness of $\cO_{\cX}(m\cL)$, we know that $h^0(X,\cO_X(mL))=h^0(\cX_b, \cO_{\cX_b}(m))$ for any $m\geq 2-n$. On the other hand, by Serre duality for CM sheaves \cite[Theorem 5.71]{KM98}, we know that $H^i(X,\cO_X(mL))=H^i(\cX_b, \cO_{\cX_b}(m))=0$ for any $i\leq n-1$ and $h^n(X,\cO_X(mL))=h^n(\cX_b, \cO_{\cX_b}(m))$  whenever $-mL$ and $-m\cL_{b}$ are ample, i.e.\ $m\leq -1$. Thus part (3) is proven. 

For part (4), we use the local volume estimates. Let $x\in X$ be a point where $D$ is not Cartier. Denote by $\ind(x,D)$ the Cartier index of $D$ at $x$. Then by Theorems \ref{thm:vol-finitedeg}, \ref{thm:odp-lowdim}(2), and \ref{thm:liu18}, we know that 
\[
\frac{ 3 n^n (n-1)^n}{(n+1)^n}=\frac{n^n}{(n+1)^n}(-K_X)^n\leq \hvol(x,X)\leq \frac{n^n}{\ind(x,D)}.
\]
In particular, we get $\ind(x,D)\leq \frac{(n+1)^n}{3(n-1)^n}<3$. This implies that $\ind(x,D)=2$. 
\end{proof}

From now on, we restrict our focus to cubic fourfolds, i.e.\ $n=4$. We will always denote by $X$ a K-semistable $\bQ$-Fano variety admitting a $\bQ$-Gorenstein smoothing to cubic fourfolds. By \cite{SS17, LX19}, to prove Theorem \ref{thm:K=GIT} the main challenge is to show that $L$ is Cartier on $X$. This would follow from Conjecture \ref{conj:ODP} in dimension $4$ as indicated by Proposition \ref{prop:odp-violate}. However, currently we are unable to confirm Conjecture \ref{conj:ODP} for $n\geq 4$. In the below, we  provide some partial results using the local volume estimates from Section \ref{sec:k-vol}. We will study the global geometry of $(X,L)$ in Section \ref{sec:kawamata}.

\begin{prop}\label{prop:non-Cartier-finite}
With the above notation, we have that $L$ is Cartier away from a finite subset $\Sigma\subset X$. Moreover, any $x\in \Sigma$ is an isolated singularity of $X$.
\end{prop}

\begin{proof}
Let $\Sigma$ be the closed subset of $X$ where $L$ is not Cartier. Clearly, we have $\Sigma\subset X_{\sing}$. First of all, assume to the contrary that $\Sigma$ contains a curve $C$. Let $x\in C$ be a point.  Let $\tau: (\tilde{x}\in \tX)\to (x\in X)$ be the index $1$ cover with respect to $L$. Let $\tC:=\pi^{-1}(C)$.  Then by finite degree formula we know that $\hvol(\tilde{x}',\tX)/4^4\geq 2\cdot 3^5/5^4$ for any $\tilde{x}'\in \tC$. If $\tX$ is singular along $\tC$, then by Corollary \ref{cor:largesing} we know that $\hvol(\tilde{x}',\tX)/4^4\leq \frac{16}{27}$. This is a contradiction since $2\cdot 3^5/5^4>16/27$. Hence $\tX$ is smooth at the generic point of $\tC$. This implies that $(x'\in X)$ is a quotient singularity of order $2$ for a general point $x'\in C$. Since $C$ is contained in the ramification locus of $\tau$, we know that $(x'\in X)$ has type $\frac{1}{2}(1,1,1,0)$ or $\frac{1}{2}(1,1,0,0)$. 

Next, we will show that neither quotient type is possible. The argument is similar to \cite[Proof of Lemma 3.16]{LX19}. If $(x'\in X)$ has type $\frac{1}{2}(1,1,1,0)$, we pick a general hyperplane section $\cH$ through $x'$ of $\cX$ embedded in some projective space. Then clearly $\cH_0=\cH\cap X$ has a quotient singularity of type $\frac{1}{2}(1,1,1)$, while $\cH_b=\cH\cap \cX_b$ is smooth for $b\in B^\circ$. This contradicts the rigidity theorem of Schlessinger \cite{Sch71}. If $(x'\in X)$ has type $\frac{1}{2}(1,1,0,0)$, then we know that $X$ has hypersurface singularities near $x'$, and so does $\cX$. We pick two general hyperplane sections $\cH_1$ and $\cH_2$ through $x'$ of $\cX$. Then clearly $(x'\in \cH_1\cap \cH_2)$ is a normal isolated hypersurface singularity of dimension $3$. Since $\cO_{\cX}(\cL)$ is Cohen-Macaulay, there is a well-defined $\bQ$-Cartier Weil divisor class $\cL|_{\cH_1\cap\cH_2}$ such that $\cO_{\cH_1\cap\cH_2}(\cL|_{\cH_1\cap\cH_2})\cong \cO_{\cX}(\cL)\otimes \cO_{\cH_1\cap\cH_2}$. By local Grothendieck-Lefschetz theorem \cite{Rob76}, the local class group of $(x' \in  \cH_1\cap\cH_2)$ is torsion free which implies that $\cL|_{\cH_1\cap\cH_2}$ is Cartier at $x'$. Hence this implies that $\cL$ is Cartier at $x'$, which implies $L$ is also Cartier at $x'$. This is a contradiction.

Finally, we show that $\Sigma$ consists only of isolated singularities. Assume to the contrary that $x\in \Sigma$ is not isolated. Let $C'\subset X_{\sing}$ be a curve through $x$. Again, let $(\tilde{x}\in \tX)$ be the index one cover of $(x\in X)$ with respect to $L$. Since $L$ is Cartier at the generic point of $C'$, we know that $\tC':=\pi^{-1}(C')$ is contained in $\tX_{\sing}$. Hence by Theorems  \ref{thm:vol-finitedeg}, \ref{thm:liu18}, and Corollary \ref{cor:largesing} we know that $2\cdot 3^5/5^4\leq\hvol(\tilde{x},\tX)/4^4\leq \frac{16}{27}$, again a contradiction.
\end{proof}

\begin{prop}\label{prop:odp-violate}
With the above notation, if $L$ is not Cartier at $x\in X$, then the index $1$ cover $\tilde{x}\in \tX$ violates  Conjecture \ref{conj:ODP}.
\end{prop}

\begin{proof}
By Theorems \ref{thm:vol-finitedeg} and \ref{thm:liu18}, we have
\[
\frac{ 3^5\cdot  4^4 }{5^4}=\frac{4^4}{5^4}(-K_X)^4\leq \hvol(x,X) =  \frac{\hvol(\tilde{x},\tX)}{\ind(x,L)}.
\]
Since $2L$ is Cartier at $x\in X$ by Theorem \ref{thm:LX-constrain}, we know that $\ind(x,L)=2$. Hence 
\[
\hvol(\tilde{x},\tX)\geq \frac{2\cdot 3^5\cdot  4^4}{5^4}> 2 \cdot 3^4. 
\]
If $\tilde{x}\in \tX$ is smooth, then Proposition \ref{prop:non-Cartier-finite} implies that $\tau: \tX\to X$ is ramified only at $x$ which implies that $(x\in X)$ is an isolated quotient singularity of order $2$ admitting a $\bQ$-Gorenstein smoothing. This contradicts \cite{Sch71}. Hence  $\tilde{x}\in \tX$ violates Conjecture \ref{conj:ODP}. 
\end{proof}

\section{Ambro-Kawamata non-vanishing approach}\label{sec:kawamata}

In this section, we use the following non-vanishing theorem of Ambro \cite[Main Theorem]{Amb99} and Kawamata  \cite[Theorem 5.1]{Kaw00} to study the geometry of K-semistable $\bQ$-Gorenstein limits of cubic fourfolds.  

\begin{thm}[\cite{Amb99, Kaw00}]\label{thm:kawamata}
Let $(Y,\Delta)$ be a projective klt pair. Let $M$ be a nef Cartier divisor over $Y$ such that $M-K_Y-\Delta$ is nef and big. Assume that there exists a rational number $r>\dim(Y)-3\geq 0$ such that $-K_Y-\Delta\sim_{\bQ} rM$. Then $H^0(Y,M)\neq 0$, and for a general member $D\in |M|$ the pair $(Y,\Delta+D)$ is plt. 
\end{thm}

In the rest of this paper, we adapt the notation of Theorem \ref{thm:LX-constrain} and assume $n=4$. In particular, $X$ is a K-semistable $\bQ$-Gorenstein limit of cubic fourfolds, and $L$ is an ample $\bQ$-Cartier Weil divisor on $X$ such that $-K_X\sim 3L$. Denote by $\Sigma$ the non-Cartier locus of $L$ on $X$ which is a finite set by Proposition \ref{prop:non-Cartier-finite}. 

Next, we apply Ambro-Kawamata's non-vanishing theorem to our study on the geometry of linear systems $|2L|$ and $|L|$. Our goal is to show that $L$ is Cartier on $X$.

\begin{prop}\label{prop:2Ldisjoint}
Let $D_1,D_2$ be two general member of $|2L|$ on $X$. Then both $D_i~(i=1,2)$ and their complete intersection $S:=D_1\cap D_2$ are Gorenstein canonical. Moreover, $\Bs|2L|$ is disjoint from $\Sigma$.
\end{prop}

\begin{proof}
We first show that both $D_i$ and $S$ are klt of Gorenstein index at most $2$. By Theorem \ref{thm:LX-constrain}, we know that $2L$ is Cartier and ample. By applying Theorem \ref{thm:kawamata} to $2L$ on $X$, we see that $-K_X\sim_{\bQ}\frac{3}{2}(2L)$, so $r=\frac{3}{2}> 1=\dim(X)-3\geq 0$. So we have that $(X,D_i)$ is plt, hence $D_i$ is klt. Next, we apply Theorem \ref{thm:kawamata} to $2L|_{D_i}$ on $D_i$. By adjunction it is clear that $-K_{D_i}\sim L|_{D_i}$, so $r=\frac{1}{2}\geq 0$. Also, by Theorem \ref{thm:LX-constrain}(3) we have an exact sequence 
\[
H^0(X,2L)\to H^0(D_i,2L|_{D_i})\to H^1(X,\cO_X)=0.
\]
Hence the general divisor $D_2\in |2L|$ restricts to a general divisor $S\in |2L|_{D_1}|$. In particular, $S$ is also klt. By adjunction, we know that $K_S\sim L|_S$, so both $D_i$ and $S$ have Gorenstein index at most $2$.

Next we show that $S$ is Gorenstein. 
Assume to the contrary that $x\in S$ has Gorenstein index $2$. Then clearly $x\in \Sigma$. 
Let $\tau:(\tilde{x}\in \tX)\to (x\in X)$ be the index $1$ cover of $L$. Since $\tX$ is Gorenstein, the preimage $\tS:=\tau^{-1}(S)$ is also Gorenstein as it is a complete intersection in $\tX$. Thus $\tilde{x}\in \tS$ is a Du Val singularity. This implies that $\edim(\tilde{x},\tX)\leq 5$, i.e.\ $\tilde{x}\in \tX$ is a hypersurface singularity. But then the ODP conjecture holds for $\tilde{x}\in \tX$ by Theorem \ref{thm:lci},  and we get a contradiction to Proposition \ref{prop:odp-violate}. This shows that $S$ is Gorenstein. 

Since both $D_i$ and $S$ are Cohen-Macaulay, we know that $\cO_X(L)\otimes \cO_S\cong \cO_S(L|_S)\cong \omega_S$. In particular, this implies that $S\cap \Sigma=\emptyset$. It is clear that $\Bs|2L|\subset S$, so $\Bs|2L|\cap \Sigma=\emptyset$.
\end{proof}

\begin{prop}\label{prop:2LK3}
Let $D\in |2L|$ and $H\in |L|$ be general divisors on $X$. Let $G:= D\cap H$ be their complete intersection. Then $(G, L|_G)$ is a polarized K3 surface with Du Val singularities of degree $6$. Moreover, $|2L|$ is base point free, and the connected components of $\Bs|L|$ have dimension $0$ or $2$.
\end{prop}

\begin{proof}
By Proposition \ref{prop:2Ldisjoint}, we know that $L|_D\sim -K_D$ is Cartier. Since $\cO_X(-L)$ and  $\cO_X(L)$ are both Cohen-Macaulay by Theorem \ref{thm:LX-constrain}(1), we know that $\cO_D(L|_D)\cong \cO_X(L)/\cO_X(-L)$. Hence we have an exact sequence $H^0(X,L)\to H^0(D,L|_D)\to H^1(X,-L)=0$ by Theorem \ref{thm:LX-constrain}(3). Hence $G$ is a general divisor in $|L|_D|$ which implies that $(D, G)$ is plt by Theorem \ref{thm:kawamata}. Since $G\sim L|_D$ is Cartier and $D$ is Gorenstein canonical, we know that $G$ has Gorenstein canonical singularities as well. By adjunction, $K_G\sim 0$. We claim that $H^1(G,\cO_G)=0$. Since $L|_D$ is Cartier, there is an exact sequence
\[
H^1(D, \cO_D)\to H^1(G,\cO_G)\to H^2(D, -L|_D). 
\]
Thus it suffices to show that both $H^1(D, \cO_D)$ and $H^2(D, -L|_D)$ vanish. Since $D$ is Gorenstein canonical, we have $\cO_D(mL|_D)\cong \cO_X(mL)\otimes \cO_D$. Hence we have the following exact sequences for $i=1,2$:
\[
0=H^i(X, mL)\to  H^i(D, mL|_D)\to H^{i+1}(X, (m-2)L)=0.
\]
Here we use Theorem \ref{thm:LX-constrain}(3). Thus both $H^1(D, \cO_D)$ and $H^2(D, -L|_D)$ vanish, and the claim follows. Hence $G$ is a K3 surface with Du Val singularities. The polarization $L|_G$ has degree $6$ since $(L|_G^2)=(L^2\cdot D\cdot H)=2(L^4)=6$. 

Since $\Bs|2L|\subset \Bs|L|$, we know that $\Bs|2L|\subset G$. Moreover, we can show that $H^0(X,2L)\to H^0(G, 2L|_G)$ is surjective by tracing exact sequences as below and using Theorem \ref{thm:LX-constrain}(3):
\[
H^0(X,2L)\to H^0(D, 2L|_D)\to H^1(X, \cO_X)=0,
\]
\[
H^0(D, 2L|_D)\to H^0(G, 2L|_G) \to H^1(D, L|_D)=0.
\]
Hence $\Bs|2L|=\Bs |2L|_G|$. By classical result on linear system of K3 surfaces (see \cite[Remark 3.4]{huybrechts}), we know that $2L|_G$ is always base point free. Hence $|2L|$ is also base point free.

By similar argument on tracing exact sequences, we know that $H^0(X,2L)\to H^0(H,2L|_H)$ is surjective. Since $G$ is a general divisor of the ample base point free linear system $|2L|_H|$ on $H$, we know that $H$ is integral which implies that $|L|$ has no fixed component. By tracing exact sequences as above, we know that $H^0(X,L)\to H^0(H, L|_H)\to H^0(G,L|_G)$ are both surjective. By Mayer's theorem \cite{May72} (see also \cite[Corollary 3.15]{huybrechts}), we know that $L|_G$ is either base point free or has a fixed component isomorphic to $\bP^1$. Thus any connected component of $\Bs|L|=\Bs|L|_H|$ is either an isolated point or a surface.
\end{proof}

\begin{prop}\label{prop:H-degenerate}
Let $H\in |L|$ be a general divisor. Then $H$ is Gorenstein log canonical. Moreover, $H$ admits a weakly special test configuration with central fiber isomorphic to the projective cone $C_p(G, 2L|_G)$.
\end{prop}

\begin{proof}
By the proof of Proposition \ref{prop:2LK3}, we know that $H$ is integral. Since the ideal sheaf of $H$ in $X$ is $\cO_X(-L)$, which is Cohen-Macaulay by Theorem \ref{thm:LX-constrain}(1), we know that $H$ is Cohen-Macaulay as well. Since a general member $G$ of the base point free linear system $|2L|_H|$ is normal, we know that $H$ is $R_1$ hence normal as well. By adjunction, we have that $K_H=(K_X+H)|_H\sim -2L|_H$ is Cartier. Hence $H$ is Gorenstein normal. 

Next we construct the weakly special test configuration. The idea is by degeneration to the normal cone of $G$. Let $R=\oplus_{m=0}^{\infty} R_m := H^0(H, 2mL|_H)$ be the section ring of $(H, 2L|_H)$. Consider the $\bN$-filtration $\cF$ on $R$ (see e.g.\ \cite[Section 2.3]{BX18} for backgrounds) as 
\[
\cF^p R_m := H^0(H, 2mL|_H - pG)\subset H^0(H, 2mL|_H)=R_m\quad \textrm{ if $p\in \bZ_{\geq 0}$}.
\]
For $p\in \bZ_{<0}$ we define $\cF^p R_m=R_m$.
Since $G\sim 2L|_H$, it is clear that $\cF^\bullet R$ is a multiplicative, linearly bounded, finitely generated $\bN$-filtration of $R$. Denote by 
\[
\cH:=\Proj_{\bA^1} \bigoplus_{m=0}^{+\infty}\bigoplus_{p=-\infty}^{+\infty} t^{-p}\cF^{p} R_m,
\]
where $t$ is the parameter of $\bA^1$, and the grading of $\cF^p R_m$ is $m$. Then  $(\cH, \cO_{\cH}(1))\to \bA^1$ is a test configuration of $(H,2L|_H)$. The central fiber $\cH_0$ is given by 
\[
\cH_0= \Proj \bigoplus _{m=0}^{+\infty}\bigoplus_{p=-\infty}^{+\infty} \cF^{p} R_m/\cF^{p+1} R_m.
\]
It is clear that $\cF^{p} R_m/\cF^{p+1} R_m=0$ for $p<0$. 
Hence to show $\cH_0\cong C_p(G, 2L|_G)$ it suffices to show that $\cF^{p} R_m/\cF^{p+1} R_m\cong H^0(G, 2(m-p)L|_G)$ for $p\geq 0$. By tracing the exact sequence
\begin{align*}
0&\to H^0(H, 2mL|_H - (p+1)G)\to H^0(H, 2mL|_H - pG)\to H^0(G, 2(m-p)L|_G)\\ & \to H^1(H, 2mL|_H - (p+1)G) \cong H^1(H, 2(m-p-1)L|_H),
\end{align*}
it suffices to show that $H^1(H, 2qL|_H)=0$ for any $q\in \bZ$. This follows from the following exact sequence and Theorem \ref{thm:LX-constrain}(3):
\[
0=H^1(X, 2qL)\to H^1(H, 2qL|_H)\to H^2(X, (2q-1)L)=0.
\]
Hence we have shown $\cH_0\cong C_p(G, 2L|_G)$. Since $G$ is a K3 surface with canonical singularities, we know that $C_p(G, 2L|_G)$ is log canonical by \cite[Lemma 3.1]{Kol13}. By inversion of adjunction, we know that $(\cH, \cH_0)$ is log canonical near $\cH_0$, which implies that $(\cH, \cH_0)$ is log canonical as $\cH\setminus\cH_0\cong H\times(\bA^1\setminus\{0\})$.  Hence $H$ is log canonical, and $\cH$ is a weakly special test configuration of $H$. The proof is finished.
\end{proof}

Next, we divide the argument into cases based on the geometry of the polarized K3 surface $(G,L|_G)$ where $G$ is a general complete intersection of $D\in |2L|$ and $H\in |L|$. By Mayer's theorem \cite{May72} (see also \cite[Remark 3.8 and Corollary 3.15]{huybrechts}), there are three cases based on the behavior of the linear system $|L|_G|$.
\begin{enumerate}
    \item (\textit{unigonal}) $|L|_G|$ has a base curve $C_0\cong \bP^1$, $|L|_G - C_0|$ is base point free, and $\phi_{|L|_G - C_0|}: G\to \bP^4$ is an elliptic fibration over a quartic rational normal curve.
    \item (\textit{hyperelliptic}) $|L|_G|$ is base point free, and $\phi_{|L|_G|}: G\rightarrow \bP^4$ is a double cover onto a non-degenerate rational surface of degree $3$ in $\bP^4$.
    \item (\textit{complete intersection}) $|L|_G|$ is very ample, and $\phi_{|L|_G|}: G\hookrightarrow \bP^4$ embeds $G$ as a $(2,3)$-complete intersection in $\bP^4$.
\end{enumerate}

The next result shows that $(G,L|_G)$ cannot be unigonal.

\begin{prop}\label{prop:K3-unigonal}
The polarized K3 surface $(G,L|_G)$ is not unigonal. In particular, $\Bs|L|$ is a finite set.
\end{prop}

\begin{proof}
Assume to the contrary that $(G,L|_G)$ is unigonal of degree $6$. Then we know that $L|_G\sim 4C_1 +C_0$ where $C_0\cong\bP^1$ is a $(-2)$-curve, and $C_1$ is a general fiber of the elliptic fibration $G\to \bP^1$ induced by $|L|_G -C_0|$. In the proof of Proposition \ref{prop:2LK3} we have shown that  $H^0(X,L)\to H^0(H, L|_H)\to H^0(G, L|_G)$ are both surjective.  Indeed, by the following exact sequence
\[
0=H^0(H,-L|_H) \to H^0(H, L|_H)\to H^0(G, L|_G)\to H^1(H, \cO_H)=0,
\]
we know that $H^0(H, L|_H)\cong H^0(G, L|_G)$.


Next, we resolve the birational map $\phi_{|L|}:X\dashrightarrow \bP^5$ as follows:
\begin{center}
\begin{tikzcd}
& X' \arrow[ld,swap,"\pi"] \arrow[rd,"\rho"] &  \\
X \arrow[rr, dashed,"\phi_{|L|}"] && \bP^5\\
\end{tikzcd}
\end{center}
Here $X'$ is the normalization of the graph of $\phi_{|L|}$. From the above discussion, we know that for a general divisor $H\in |L|$, the image $\rho(\pi_*^{-1} H)$ is a general hyperplane section of $W:=\rho(X')$. From the above surjectivity between $H^0$'s, we know that the restrictions of $\phi_{|L|}$ to $H$ and $G$ are $\phi_{|L|_H|}$ and $\phi_{|L|_G|}$ respectively. We claim that $\rho(\pi_*^{-1} H)$ is a curve.

Assume to the contrary that $\dim(\rho(\pi_*^{-1} H))\geq 2$. Denote by $|L|_H|=E+\Lambda_H$ where $E$ is the base component and $\Lambda_H$ is movable. Then it is clear that $\rho(\pi_*^{-1} H)=\phi_{\Lambda_H}(H)$. Since $\dim(\phi_{\Lambda_H}(H))\geq 2$, Bertini's theorem implies that a general member $F\in \Lambda_H$ is an integral surface. Since $G$ is ample on $H$, we know that $F\cap G$ is connected.  Since $\Bs|L|_G|=C_0$ and $G$ is a general member of the base point free linear system $|2L|_H|$, we know that $E|_G=C_0$. Hence $F|_G$ is a general member of $|L|_G-C_0|$ which is the sum of four distinct elliptic fibers. In particular, $F\cap G$ is disconnected. This is a contradiction. Thus the claim is proved.

Since $\rho(\pi_*^{-1} H)$ is a curve, it is the same as $\phi_{|L|_G|}(G)$ which is a rational normal curve of degree $4$ in $\bP^4$. Since $\rho(\pi_*^{-1} H)$ is a hyperplane section of $W$, we know that $W\subset\bP^5$ is a non-degenerate surface of degree $4$. By the classification of minimal degree varieties (see e.g.\ \cite{EH87}), we know that $W$ is isomorphic to either $\bP(1,1,4)$ (cone over a quartic rational normal curve, $\phi_{|\cO(4)|}:\bP(1,1,4)\hookrightarrow\bP^5$), $\bF_{2,1}$ ($\phi_{|3f+e|}: \bF_2\hookrightarrow\bP^5$ where $f$ and $e$ are a fiber and the negative section respectively), $\bF_{0,2}$ ($\phi_{|\cO(1,2)|} :\bP^1\times\bP^1\hookrightarrow\bP^5 $), or the Veronese surface $V_4$ ($\phi_{|\cO(2)|}:\bP^2\hookrightarrow\bP^5$). In each case, there exists a family $\{\cC_t\}$ of non-reduced divisors in the linear system $|\cO_W(1)|$ that covers $W$. Choose a general divisor $\cC_t$, then we know $\pi_*\rho^* \cC_t$ is a non-reduced divisor in $|L|$. Since $-K_X\sim 3L$, we know that $\alpha(X)\leq \frac{1}{6}$ which implies that $X$ is K-unstable by Theorem \ref{thm:alpha-lowerbound}. This is a contradiction. The conclusion on $\Bs|L|$ follows from the previous discussion since $|L|_G|$ is base point free in the non-unigonal cases.
\end{proof}

Next we treat the hyperelliptic case. 


\begin{prop}\label{prop:K3-hyp}
Assume $(G,L|_G)$ is hyperelliptic. Then $X$ is K-unstable. 
\end{prop}

\begin{proof}
We resolve the rational map $\phi_{|L|}=\rho\circ \pi^{-1}$ by $X\xleftarrow{\pi}X'\xrightarrow{\rho}W$ where $X'$ is the normalization of the graph of $\phi_{|L|}$. Denote by $\pi^*|L|=\frac{1}{2}E+\Lambda$ where $E$ is an effective Weil divisor on $X'$ and $\Lambda$ is base point free. Let $L':= \pi^*L - \frac{1}{2}E$ which is semiample on $X'$. Since $\rho(\pi_*^{-1} H)$ is a hyperplane section of $W$, and $\rho|_{G'}$ is a double cover for general $G$ and $G':=\pi^{-1}(G)$, we know that $\dim(W)\geq 3$. We first show that $\dim(W)=3$.

Assume to the contrary that $\dim(W)=4$, i.e.\ $\rho$ is generically finite. Since $\Bs|L|$ is a finite set, we know that $(\pi^* L\cdot E)=0$ as a $2$-cycle. Hence 
\[
\deg(W)\cdot \deg(\rho)=(L'^4)\leq (L'^3\cdot \pi^* L)=((\pi^* L-\tfrac{1}{2}E)^3\cdot \pi^*L)=((\pi^*L)^4)=3.
\]
Since $W$ is non-degenerate, we have $\deg(W)\geq 2$ which implies that $\deg(\rho)=1$, i.e.\ $\rho$ is birational. Let $H':=\pi_*^{-1}H$ be a general divisor in $\Lambda$. Then we know that $\rho|_{H'}$ is also birational. However, $\rho|_{G'}$ has degree $2$ for a general $G'$ in the base point free linear system $(\pi^*|2L|)|_{H'}$. This is a contradiction. Thus we have $\dim(W)=3$.

Next, we analyze the geometry of $W$. Notice that since $(G,L|_G)\cong (G', (\pi^*L)|_{G'})$ has degree $6$, the image $\rho(G')$ is an integral surface of degree $3$. Since $\dim(W)=3$, we know that $\rho(H')$ is an integral surface for general $H'\in \Lambda$, hence $\rho(H')=\rho(G')$ has degree $3$. This implies that $\deg(W)=3$ as well. Now $W$ is a non-degenerate threefold in $\bP^5$ of degree $3$. By \cite{EH87}, there are three possibilities of $W$: $\bP(1,1,3,3)$ (a second iterated cone over a twisted cubic curve, $\phi_{|\cO(3)|}: \bP(1,1,3,3)\hookrightarrow\bP^5$), the cone over $\bF_{1,1}$ (the cone over the image of $\phi_{|2f+e|}:\bF_1\hookrightarrow \bP^4$), or $\bP^1\times\bP^2$ ($\phi_{|\cO(1,1)|}:\bP^1\times\bP^2\hookrightarrow\bP^5$). In the first two cases, $W$ is covered by non-reduced hyperplane sections, which implies that $|L|$ contains a non-reduced element by similar arguments to the proof of Proposition \ref{prop:K3-unigonal}. This shows $\alpha(X)\leq \frac{1}{6}$ which implies that $X$ is K-unstable by Theorem \ref{thm:alpha-lowerbound}.




The only case left is when $(W,\cO_W(1))\cong (\bP^1\times\bP^2, \cO(1,1))$. We will show that this case cannot occur. Our argument is inspired by \cite[Section 6]{dFH11}.\footnote{This argument is suggested by Chenyang Xu.}
From the proof of Proposition \ref{prop:2LK3}, we see that $\phi_{|L|}$ restricted to $D$ is the finite morphism $\phi_{|L|_D|}:D\to W$ as a double cover. Thus $\phi_{|L|_D|}^*:\Pic(W)\hookrightarrow \Pic (D)$ is an injection, which implies that $\rk\Pic(D)\geq \rk\Pic(W)=2$. Since $X$ admits a $\bQ$-Gorenstein smoothing $f: \cX\to B$ with $\cX_0\cong X$ where $\cX_b$ is a smooth cubic fourfold for any $b\in B^\circ$, we know that $f_*\cO_{\cX}(2\cL)$ is flat by Theorem \ref{thm:LX-constrain}(3). Hence after base change to a quasi-finite holomorphic map $\bD\to B$ from the unit disc $\bD\subset \bC$, we can find a $\bQ$-Gorenstein smoothing $\cD\to \bD$ of $D\cong \cD_0$ such that $\cD_t$ is a smooth $(2,3)$-complete intersection in $\bP^5$ for any $t\in \bD^\circ:=\bD\setminus \{0\}$. By Proposition \ref{prop:2Ldisjoint} we know that $\cD_0$ is a $\bQ$-Fano variety with Gorenstein canonical singularity. Hence by Kawamata-Viehweg vanishing we have $H^i(\cD_0, \cO_{\cD_0})=0$ for any $i>0$. Similarly, since $\cD\to \bD$ is a $\bQ$-Gorenstein flat family of $\bQ$-Fano varieites, we have $H^i(\cD, \cO_{\cD})=0$ for any $i>0$. Hence from the exponential exact sequence, we obtain the following isomorphisms:
\[
\Pic(\cD_0)\xrightarrow{\cong }H^2(\cD_0,\bZ) \xleftarrow{\cong }H^2(\cD,\bZ)\xleftarrow{\cong }\Pic(\cD).
\]
Here the middle isomorphism follows from the topological fact that $\cD_0\hookrightarrow\cD$ admits a deformation retraction. In particular, we know that $\rk\Pic(\cD)=\rk\Pic(\cD_0)\geq 2$. By Lefschetz hyperplane theorem, every fiber $\cD_t$ of the smooth fibration $\cD^\circ \to \bD^\circ$ satisfies that $\Pic(\cD_t)=\bZ\cdot [-K_{\cD_t}]$. Hence \cite[Conditions 12.2.1]{KM92} hold for $\cD\to \bD$. By \cite[12.2.2 -- 12.2.5]{KM92}, we know that there is a $\bQ$-local system $\cG\cN^1(\cD/\bD)$ on $\bD$ satisfying that $\cG\cN^1(\cD/\bD)(\bD)\cong \Pic(\cD)\otimes_{\bZ}\bQ$ and $\cG\cN^1(\cD/\bD)|_t\cong \Pic(\cD_t)\otimes_{\bZ}\bQ$ for a very general $t\in \bD^\circ$. Since $\bD$ is contractible, we know that $\cG\cN^1(\cD/\bD)$ is a trivial $\bQ$-local system, and 
\[
2\leq \rk\Pic(\cD)=\rk \cG\cN^1(\cD/\bD) = \rk\Pic(\cD_t)=1.
\]
This is a contradiction. Thus the proof is finished.
\end{proof}

Finally we treat the case of $G$ being a complete intersection.

\begin{prop}\label{prop:K3-ci}
Suppose $L$ is not Cartier at $x\in X$. 
Assume $(G,L|_G)$ is a $(2,3)$-complete intersection in $\bP^4$. Then the index $1$ cover $\tilde{x}\in \tX$ of $x\in X$ with respect to $L$ is a local complete intersection singularity.
\end{prop}

\begin{proof}
It is clear that $x\in \Bs|L|\subset H$. By Proposition \ref{prop:H-degenerate}, we know that there exists a weakly special test configuration $\cH$ of $H$ with central fiber $\cH_0\cong C_p(G, 2L|_G)$. Denote by $\cL_{\cH}$ the Zariski closure of $L|_{H}\times(\bA^1\setminus \{0\})$ in $\cH$. Then it is clear that $\cL_{\cH}$ is a $\bG_m$-invariant $\bQ$-Cartier Weil divisor on $\cH$.
Let $\cL_0$ be the restriction of $\cL_{\cH}$ on $\cH_0$ which is also a $\bQ$-Cartier Weil divisor. From the construction of $\cH$, we know that the Zariski closure $\cG$ of $G\times(\bA^1\setminus\{0\})$ in $\cH$ is a trivial test configuration of $G$. Moreover, its central fiber $\cG_0$ is precisely the section at infinity in the projective cone $C_p(G,2L|_G)$. Thus we have $\cL_0|_{\cG_0}\cong L|_G$ under the natural identification of $\cG_0\cong G$. Since $\cL_0$ is $\bG_m$-invariant, we know that $\cL_0$ is linearly equivalent to the cone over $L|_G$. In particular, we know that $\ind(o, \cL_0)=\ind(o,\cL_{\cH})=2$ where $o\in \cH_0$ is the cone vertex. Besides, since $L|_{G}$ is Cartier, we know that $o$ is the only non-Cartier point of $\cL_{\cH}$ in $\cH_0$.

From earlier discussions, we know that $L|_H$ is not Cartier at $x$. Thus $\cL_{\cH}$ is not Cartier at $(x,t)$ for any $t\in \bA^1\setminus \{0\}$. This implies that the degeneration of $(x,t)$ in $\cH$ as $t\to 0$ is precisely $o$. Let $\tau_{\cH}: (\tilde{o}\in \widetilde{\cH})\to (o\in \cH)$ be the index $1$ cover of $\cL_{\cH}$. Then it is clear that $\tilde{o}\in \widetilde{\cH}_0$ is isomorphic to the affine cone singularity $C_a(G, L|_G)$. Since $G$ is a global complete intersection and $L|_G\cong \cO_G(1)$, we know that $\tilde{o}\in \widetilde{\cH}_0$ is a local complete intersection singularity. We denote by $\tau:(\tilde{x}\in \tX)\to (x\in X)$ the index $1$ cover of $L$. Denote by $\widetilde{H}:=\tau^{-1}(H)$. Then it is clear that $\widetilde{\cH}$ provides a $\bG_m$-equivariant degeneration of $(\tilde{x}\in \tH)$ to $(\tilde{o}\in \widetilde{\cH}_0)$ which is a local complete intersection singularity. By \cite[Theorem 2.3.4]{CMrings}, we know that $(\tilde{x}\in \tH)$ is a local complete intersection singularity. Since $\tH$ is a Cartier divisor of $\tX$, again using \cite[Theorem 2.3.4]{CMrings} we conclude that $\tilde{x}\in \tX$ is also a local complete intersection. The proof is finished.
\end{proof}

\begin{rem}\label{rem:zhuang} (\textit{Communicated with Ziquan Zhuang.}) There is an alternative way to prove Propositions \ref{prop:K3-hyp} and \ref{prop:K3-ci} using higher codimensional $\alpha$-invariants. Since $X$ is K-semistable and $\Bs|L|$  is finite, by \cite[Theorem 1.1]{Zhu20b} we know that $\alpha^{(4)}(X)\geq \frac{4}{5}$. Since $-K_X\sim 3L$, we know that   $\lct(X; |L|)\geq \frac{12}{5}$. Hence $(X, H_1+H_2+\frac{2}{5}H_3)$ is log canonical for general members $H_i\in |L|~(1\leq i\leq 3)$. Suppose $x\in \Sigma$ is a non-Cartier point of $L$ with the index $1$ cover $\tau:(\tilde{x}\in \tX)\to (x\in X)$. Thus $(\tX, \tH_1+\tH_2+\frac{2}{5}\tH_3)$ is log canonical at $\tilde{x}$ where each $\tH_i:=\tau^* H_i$ is Cartier. Hence by adjunction we know that $(\tH_1\cap \tH_2, \frac{2}{5}\tH_1\cap\tH_2\cap\tH_3)$ is semi-log-canonical (slc). Hence $\tilde{x}\in \tH_1\cap \tH_2$ is a Gorenstein slc surface singularity without an isolated lc center. By the classification of log canonical surface singularities (see e.g.\ \cite[Section 3.3]{Kol13}), we know that $\tilde{x}\in \tH_1\cap \tH_2$ is either Du Val or nodal. In particular, $\tilde{x}\in \tH_1\cap \tH_2$ is a hypersurface singularity which implies that $\tilde{x}\in \tX$ is also a hypersurface singularity. Then similar arguments to the proof of Theorem \ref{thm:L-Cartier} implies that $X$ is K-unstable, a contradiction.
\end{rem}

To summarize, we have shown the following result which implies Theorem \ref{thm:K=GIT}.

\begin{thm}\label{thm:L-Cartier}
Let $(X,L)$ be the K-semistable limit of cubic fourfolds as in Theorem \ref{thm:LX-constrain}. Then $L$ is a very ample Cartier divisor, and $\phi_{|L|}:X\hookrightarrow \bP^5$ embeds $X$ as a (possibly singular) cubic fourfold.
\end{thm}

\begin{proof}
Assume to the contrary that $L$ is not Cartier at $x\in X$. Then from the above discussions, we see that $(G, L|_G)$ is a polarized K3 surface with Du Val singularities of degree $6$. If  $(G, L|_G)$ is hyperelliptic or unigonal, then Propositions \ref{prop:K3-unigonal} and \ref{prop:K3-hyp} imply that $X$ is K-unstable, a contradiction. If $(G,L|_G)$ is a $(2,3)$-complete intersection in $\bP^4$, then Propositions \ref{prop:odp-violate} and \ref{prop:K3-ci} contradict each other since Conjecture \ref{conj:ODP} holds for local complete intersections by Theorem \ref{thm:lci}. Hence $L$ must be Cartier on $X$. The rest of the statement directly follows from \cite{Fuj90}. 
\end{proof}

\begin{proof}[Proof of Theorem \ref{thm:K=GIT}]
The proof is almost the same as \cite[Proof of Theorem 1.1]{LX19}, with the following small modifications. By \cite[Theorem 6.1]{Nad90}, \cite[p 85--87]{Tia00}, \cite{AGP06},  \cite[Theorem 1.5]{LZ20}, and \cite[Corollary 1.4]{Zhu20}, there exists at least one smooth K-stable cubic fourfolds. We also replace \cite[Lemma 3.17]{LX19} by Theorem \ref{thm:L-Cartier}. Then the proof proceeds exactly the same as \cite[Proof of Theorem 1.1]{LX19}.
\end{proof}

\begin{proof}[Proof of Corollary \ref{cor:K-cubic}]
For parts (1) and (2), by \cite[Theorem 1.1]{Laz09} we know that cubic fourfolds with simple singularities are GIT stable. Hence the statements follow from Theorem \ref{thm:K=GIT}. Part (3) follows directly from Theorem \ref{thm:K=GIT}. For part (4), Theorem \ref{thm:K=GIT} implies that any GIT semistable cubic fourfold is K-semistable,  hence it has klt singularities. A hypersurface with klt singularities must be Gorenstein canonical. The existence of (weak) KE metrics in (1)(2)(3) follows from the Yau-Tian-Donaldson Conjecture in the smooth case \cite{CDS15, Tia15} and the general case \cite{BBJ18, LTW19, Li19, LXZ21}. Thus the proof is finished.
\end{proof}

\bibliographystyle{alpha}
\bibliography{ref}

\end{document}